\documentclass[12pt]{amsart}
\def\color[#1]#2{}

\usepackage{color}

\DeclareMathOperator{\Sz}{Sz}
\DeclareMathOperator{\aq}{\alpha_q}
\DeclareMathOperator{\bq}{\beta_q}
\DeclareMathOperator{\aqp}{\alpha_{q'}}
\DeclareMathOperator{\bqp}{\beta_{q'}}
\DeclareMathOperator{\PSL}{PSL}
\DeclareMathOperator{\PGL}{PGL}
\DeclareMathOperator{\PGammaL}{P\Gamma L}
\DeclareMathOperator{\Ree}{R}
\DeclareMathOperator{\Reeb}{^2G_2}

\newcommand{\cal}{\mathcal}

\newtheorem{theorem}{Theorem}

\newtheorem{lemma}{Lemma}

\newcommand{\Aut}{\mathrm{Aut}}

\parskip\medskipamount

\title[Ree groups in characteristic 3 acting on polytopes]{Groups of Ree type in characteristic 3 acting on polytopes}
\author{Dimitri Leemans}
\address{D. Leemans\\
 Department of Mathematics\\University of Auckland\\
  Private Bag 92019\\ Auckland, New Zealand}
\email{d.leemans@auckland.ac.nz}
\author{Egon Schulte}
\address{E. Schulte\\Northeastern University,
Department of Mathematics,
360 Huntington Avenue,
Boston, MA 02115, USA}
\email{schulte@neu.edu}

\author{Hendrik Van Maldeghem}
\address{H. Van Maldeghem, Vakgroep Wiskunde
Universiteit Gent,
Krijgslaan 281, S22,
9000 Gent, Belgium}
\email{hvm@cage.ugent.be}

\date{\today}
\begin{document}
\maketitle
\begin{abstract}
Every Ree group $\Ree(q)$, with $q\neq 3$ an odd power of 3, is the automorphism group of an abstract regular polytope, and any such polytope is necessarily a regular polyhedron (a map on a surface). However, an almost simple group $G$ with $\Ree(q) < G \leq \mathsf{Aut}(\Ree(q))$ is not a C-group and therefore not the automorphism group of an abstract regular polytope of any rank. 
\end{abstract}

\section{Introduction}
Abstract polytopes are certain ranked partially ordered sets. A polytope is called ``regular'' if its automorphism group acts (simply) transitively on (maximal) flags. It is a natural question to try to classify all pairs $(\mathcal{P},G)$, where $\mathcal{P}$ is a regular polytope and $G$ is an automorphism group acting transitively on the flags of $\mathcal{P}$.  An interesting subclass is constituted by the pairs $(\mathcal{P},G)$ with $G$ almost simple, as then a lot of information is available about the maximal subgroups, centralizers of involutions, etc., of these groups, making classification possible. Potentially this could also lead to new presentations for these groups, as well as a better understanding of some families of such groups using geometry.

The study of polytopes arising from families of almost simple groups has received a lot of attention in recent years and has been very successful. We particularly refer to~\cite{Lee2005a} for almost simple groups of Suzuki type (see also~\cite{KL2009}); ~\cite{LS2005a, LS2008a,CDL2013} for groups $\PSL_2(q) \leq G \leq \PGammaL_2(q)$; ~\cite{BV2010} for groups $\PSL_3(q)$ and $\PGL_3(q)$; ~\cite{FL2011} for symmetric groups; ~\cite{FLM2012b,FLM2012} for alternating groups; and~\cite{HH2010,LM2011,LV2005} for the sporadic groups up to, and including, the third Conway group Co$_3$, but not the O'Nan group. Recently, Connor and Leemans have studied the rank $3$ polytopes of the O'Nan group using character theory~\cite{CL2013}, and Connor, Leemans and Mixer have classified all polytopes of rank at least 4 of the O'Nan group~\cite{CLM2014}.

Several attractive results were obtained in this vein, including, for instance, the proof that Coxeter's 57-cell and Gr\"unbaum's 11-cell are the only regular rank $4$ polytopes with a full automorphism group isomorphic to a group $\PSL_2(q)$ (see~\cite{LS2005a}). Another striking result is the discovery of the universal locally projective $4$-polytope of type $\{\{5,3\}_5,\{3,5\}_{10}\}$, whose full automorphism group is $J_1\times \PSL_2(19)$ (see \cite{HL2003}); this is based on the classification of all regular polytopes with an automorphism group given by the first Janko group $J_1$.

The existing results seem to suggest that polytopes of arbitrary high rank are difficult to obtain from a family of almost simple groups. Only the alternating and symmetric groups are currently known to act on abstract regular polytopes of arbitrary rank. For the sporadic groups the highest known rank is 5.

The Ree groups $\Ree(q)$, with $q=3^{2e+1}$ and $e>0$, were discovered by Rimhak Ree~\cite{Ree1960} in 1960. In the literature they are also denoted by $\Reeb(q)$. These groups have a subgroup structure quite similar to that of the Suzuki simple groups $\Sz(q)$, with $q=2^{2e+1}$ and $e>0$.  Suzuki and Ree groups play a somewhat special role in the theory of finite simple groups, since they exist because of a Frobenius twist, and hence have no counterpart in characteristic zero. Also, as groups of Lie-type, they have rank 1, which means that they act doubly transitively on sets without further structure. However, the rank 2 groups which are used to define them, do impose some structure on these sets. For instance, the Suzuki groups act on ``inversive planes''. For the Ree groups, one can define a geometry known as a ``unital''. However, these unitals, called \emph{Ree unitals}, have a very complicated and little accessible geometric structure (for instance, there is no geometric proof of the fact that the automorphism group of a Ree unital is an almost simple group of Ree type; one needs the classification of doubly transitive groups to prove this). Also, Ree groups seem to be misfits in a lot of general theories about Chevalley groups and their twisted analogues. For instance, there are no applications yet of the Curtis-Tits-Phan theory for Ree groups; all finite quasisimple groups of Lie type are known to be presented by two elements and 51 relations, except the Ree groups in characteristic 3~\cite{GKK2011}. Hence it may be clear that the Ree groups $\Ree(q)$, with $q$ a power of 3, deserve a separate treatment when investigating group actions on polytopes.

Now, the regular polytopes associated with Suzuki groups are quite well understood (see~\cite{KL2009}, \cite{Lee2005a}). But the  techniques used for the Suzuki groups are not sufficient for the Ree groups. In the present paper, we carry out the analysis for the groups $\Ree(q)$. In particular, we ask for the possible ranks of regular polytopes whose automorphism group is such a group, and we prove the following theorem.

\begin{theorem}\label{maintheo} 
Among the almost simple groups $G$ with $\Ree(q) \leq G \leq \mathsf{Aut}(\Ree(q))$ and $q = 3^{2e+1}\neq 3$, only the Ree group $\Ree(q)$ itself is a C-group. In particular, $\Ree(q)$ admits a representation as a string C-group of rank $3$, but not of higher rank. Moreover, the non-simple Ree group $\Ree(3)$ is not a C-group.
\end{theorem}

In other words, the groups $\Ree(q)$ behave just like the Suzuki groups:\ they allow representations as string C-groups, but only of rank $3$. We will describe a string C-group representation for $\Ree(q)$, $q\neq 3$, for each value of~$q$. Also, almost simple groups $\Ree(q) < G \leq \Aut(\Ree(q))$ can never be C-groups (in characteristic 3).

Rephrased in terms of polytopes, Theorem~\ref{maintheo} says that among the almost simple groups $\Ree(q) \leq G \leq \Aut(\Ree(q))$, only the groups $G:= \Ree(q)$ are automorphism groups of regular polytopes, and that these polytopes must necessarily have rank $3$.

Ree groups can also be the automorphism groups of abstract chiral polytopes. In fact, Sah~\cite{Sah69} showed that every Ree group $\Ree(3^{2e+1})$, with $2e+1$ an odd prime, is a Hurwitz group; and Jones~\cite{Jon1994} later extended this result to arbitrary simple Ree groups $\Ree(q)$, proving in particular that the corresponding presentations give chiral maps on surfaces. Hence the groups $\Ree(q)$ are also automorphism groups of abstract chiral polyhedra.

It is an interesting open problem to explore whether or not almost simple groups of Ree type also occur as automorphism groups of chiral polytopes of higher rank.

Note that the Ree groups in characteristic 2 are also very special: they are the only (finite) groups of Lie type  arising from a Frobenius twist and having rank at least 2. This makes them special, in a way rather different from the way the Ree groups in characteristic 3 are special. We think that in characteristic 2, quite geometric methods will have to be used in the study of polytopes related to Ree groups. 

\section{Basic notions}
\label{def}

\subsection{Abstract polytopes and string C-groups}

For general background on (abstract) regular polytopes and C-groups we refer to McMullen \& Schulte~\cite[Chapter 2]{MS2002}. 

A polytope ${\cal P}$ is a ranked partially ordered set whose elements are called {\it faces}. A polytope $\cal P$ of rank $n$ has faces of ranks $-1,0,\ldots,n$; the faces of ranks $0$, $1$ or $n-1$ are also called {\it vertices}, {\it edges} or {\it facets}, respectively. In particular, $\cal P$ has a smallest and a largest face, of ranks $-1$ and $n$, respectively. Each flag of $\cal P$ contains $n+2$ faces, one for each rank.  In addition to being locally and globally connected (in a well-defined sense), $\cal P$ is {\it thin\/}; that is, for every flag and every $j = 0,\ldots,n-1$, there is precisely {\it one} other ({\em $j$-adjacent\/}) flag with the same faces except the $j$-face. A polytope of rank $3$ is a {\em polyhedron\/}.  A polytope $\cal P$ is {\it regular\/} if its (automorphism) group $\Gamma ({\cal P})$ is transitive on the flags. If $\Gamma({\cal P})$ has exactly two orbits on the flags such that adjacent flags are in distinct orbits, then $\cal P$ is said to be {\it chiral\/}.  

The groups of regular polytopes are string C-groups, and vice versa.  A {\em C-group\/} of {\em rank\/} $n$ is a group $G$ generated by pairwise distinct involutions $\rho_0,\ldots,\rho_{n-1}$ satisfying the following {\it intersection property\/}:
\[  \langle \rho_j \mid j \in J\rangle \cap \langle \rho_j \mid j \in K\rangle = 
\langle\rho_j \mid j \in J\cap K\rangle \quad 
(J, K \subseteq \{0,\ldots,n-1\}).\]
Moreover, $G$, or rather $(G, \{\rho_0,\ldots, \rho_{n-1}\})$, is a {\em string\/} C-group (of rank $n$) if the underlying Coxeter diagram is a string diagram; that is, if the generators satisfy the relations
\[ (\rho_j\rho_k)^2 = 1 \quad \quad
(0 \leq j < k -1 \leq n-2).\]

Let $G_i := \langle \rho_{j} \mid j \neq i \rangle$ for each $i = 0,1,\ldots, n-1$, and let $G_{ij} := \langle \rho_{k} \mid k \neq i,j \rangle$ for each $i,j = 0,1,\ldots, n-1$ with $i\neq j$.

Each string C-group $G$ (uniquely) determines a regular $n$-polytope $\cal P$ with automorphism group~$G$. The {\em $i$-faces} of $\cal P$ are the right cosets of the distinguished subgroup $G_i$ for each $i = 0,1,\ldots, n-1$, and two faces are incident just when they intersect as cosets;   formally we must adjoin two copies of $G$ itself, as the (unique) $(-1)$- and $n$-faces of $\cal P$.  Conversely, the group $\Gamma({\cal P})$ of a regular $n$-polytope $\cal P$ is a string C-group, whose generators $\rho_j$ map a fixed, or {\em base\/}, flag $\Phi$ of $\cal P$ to the $j$-adjacent flag $\Phi^{j}$ (differing from $\Phi$ in the $j$-face).

\subsection{The Ree groups in characteristic 3}
\label{reebasics}

We let $C_k$ denote a cyclic group of order $k$ and $D_{2k}$ a dihedral group of order~$2k$. 

The Ree group $G:=\Ree(q)$, with $q=3^{2e+1}$ and $e\geq 0$, is a group of order $q^3(q-1)(q^3+1)$. It has a natural permutation representation on a Steiner system ${\cal S} := (\Omega,\mathcal{B}) = S(2,q+1,q^3+1)$ consisting of a set $\Omega$ of $q^3+1$ elements, the {\em points\/}, and a family of $(q+1)$-subsets $\mathcal{B}$ of $\Omega$, the {\em blocks\/}, such that any two points of $\Omega$ lie in exactly one block. This Steiner system is also called a {\em Ree unital\/}. In particular, $G$ acts $2$-transitively on the points and transitively on the incident pairs of points and blocks of ${\cal S}$. 

The group $G$ has a unique conjugacy class of involutions (see~\cite{Ree1960}). Every involution $\rho$ of $G$ has a  block $B$ of $\cal S$ as its set of fixed points, and $B$ is invariant under the centralizer $C_{G}(\rho)$ of $\rho$ in $G$. Moreover, $C_G(\rho)\cong C_2\times \PSL_2(q)$, where $C_2 = \langle \rho \rangle$ and the $\PSL_2(q)$-factor acts on the $q+1$ points in $B$ as it does on the points of the projective line $PG(1,q)$.

The Ree groups $\Ree(q)$ are simple except when $q=3$. In particular, $\Ree(3) \cong P\Gamma L_2(8)\cong\PSL_2(8):C_3$ and the commutator subgroup $\Ree(3)'$ of $\Ree(3)$ is isomorphic to $\PSL_2(8)$.

A list of the maximal subgroups of $G$ is available, for instance, in~\cite[p. 349]{VM2000} and~\cite{Kle1988}. Here we briefly review the list for $\Ree(q)$, with $q\neq 3$, as the maximal subgroups are required in the proof of Theorem~\ref{maintheo}; in parentheses we also note their characteristic properties relative to the Steiner system $\cal S$. 

\begin{itemize}
\item $N_G(A)\cong A : C_{q-1}$ (stabilizer of a point), where $A$ is a $3$-Sylow subgroup of $G$;
\item $C_G(\rho) \cong C_2\times \PSL_2(q)$\, (stabilizer of a block), where $C_2 = \langle \rho \rangle$ and $\rho$ is an involution of $G$;
\item $\Ree(q')$ (stabilizer of a sub-unital of $\cal S$), where $(q')^p= q$ and $p$ is a prime;
\item $N_G(A_i)$, for $i=1,2,3$, where $A_i$ is a cyclic subgroup of $G$ of one of the following kinds:\smallskip
\begin{itemize}
\item $A_1 = C_{\frac{q+1}{4}}$, with $N_G(A_1) \cong (C_{2}^{\,2}\times D_{\frac{q+1}{2}}):C_3$;
\item $A_2 = C_{q+1-3^{e+1}}$, with $N_G(A_2) \cong A_2 : C_6$;
\item $A_3 = C_{q+1+3^{e+1}}$, with $N_G(A_3) \cong A_3 : C_6$.
\end{itemize}
\end{itemize}
Note here that $q\equiv 3\bmod 8$, so $(q-1)/2$ is odd and $(q+1)/2$ is even. Moreover, since $p$ is odd, $q'-1$ and $q'+1$ divide $q-1$ and $q+1$, respectively. Finally, $q+1$ is divisible by 4 but not by 8.

The automorphism group $\mathsf{Aut}(\Ree(q))$ of $\Ree(q)$ is given by
\[ \mathsf{Aut}(\Ree(q)) \cong \Ree(q)\!:\!C_{2e+1}, \]
so in particular $\mathsf{Aut}(\Ree(3)) \cong \Ree(3)$. 

In the proof of our theorem we need the following lemma about normalizers of dihedral subgroups of dihedral groups. The proof is straightforward.

\begin{lemma}\label{norma}
Let $m,n>1$ be integers such that $m\,|\,n$. The normalizer $N_{D_{2n}}(D_{2m})$ of any subgroup $D_{2m}$ of $D_{2n}$ coincides with $D_{2m}$ if $n/m$ is odd, or is isomorphic to a subgroup $D_{4m}$ of $D_{2n}$ if $n/m$ is even. 
\end{lemma}

\section{Proof of Theorem~\ref{maintheo}}

The proof of Theorem~\ref{maintheo} is based on a sequence of lemmas. We begin in Lemma~\ref{lemma1} by showing that if $\Ree(q) < G \leq \Aut(\Ree(q))$ then $G$ can not be a C-group (with any underlying Coxeter diagram). Thus only the Ree groups $\Ree(q)$ themselves need further consideration. Then we prove in Lemma~\ref{lemma2} that $\Ree(q)$ does not admit a representation as a string C-group of rank at least $5$. In the subsequent Lemmas~\ref{lemma3},~\ref{lemma3b} and~\ref{lemma3c} we then extend this to rank $4$ and show that $\Ree(q)$ can also not be represented as a string C-group of rank $4$. Finally, in Lemma~\ref{lemma4} we construct each group $\Ree(q)$ as a rank 3 string C-group.

All information that we use about the groups $\Ree(q)$ can found in~\cite{Kle1988}. 

We repeatedly make use of the following simple observation. If $A:B$ is a semi-direct product of finite groups $A,B$ such that $B$ has odd order, then each involution in $A:B$ must lie in $A$. In fact, if $\rho=\alpha\beta$ with $\alpha\in A$, $\beta\in B$ is an involution, then $1=\rho^{2} = \alpha(\beta\alpha\beta^{-1})\beta^2$, where $\alpha(\beta\alpha\beta^{-1})\in A$ and $\beta^2\in B$; hence $\beta^{2}=1$, so $\beta=1$ and $\rho=\alpha\in A$.

\subsection{Reduction to simple groups $\Ree(q)$}

We begin by eliminating the almost simple groups of Ree type that are not simple.

\begin{lemma}\label{lemma1}
Let $\Ree(q) < G \leq \mathsf{Aut}(\Ree(q))$, where $q=3^{2e+1}$. Then $G$ is not a C-group.
\end{lemma}

\begin{proof}
Since $\mathsf{Aut}(\Ree(q)) \cong \Ree(q)\!:\!C_{2e+1}$ and $2e+1$ is odd, every involution in $\mathsf{Aut}(\Ree(q))$ lies in $\Ree(q)$ (by the previous observation), and hence any subgroup of $\mathsf{Aut}(\Ree(q))$ generated by involutions must be a subgroup of $\Ree(q)$. Thus no subgroup $G$ of $\mathsf{Aut}(\Ree(q))$ strictly above $\Ree(q)$ can be a C-group. (When $e=0$ we have $\mathsf{Aut}(\Ree(3)) \cong \Ree(3)$, so the statement holds trivially.) 
\end{proof}

\subsection{String C-groups of rank at least five}

By Lemma~\ref{lemma1} we may restrict ourselves to Ree groups $G=\Ree(q)$. We first rule out the possibility that the rank is $5$ or larger. 

\begin{lemma}~\label{lemma2}
Let $G=\Ree(q)$, where $q=3^{2e+1}\neq 3$. Suppose $G$ has a generating set $S$ of $n$ involutions such that $(G,S)$ is a string $C$-group. Then $n\leq 4$.
\end{lemma}

\begin{proof}
Let $S = \{\rho_0,\ldots ,\rho_{n-1}\}$, so in particular, $G=\langle \rho_0,\ldots ,\rho_{n-1}\rangle$. Then $\rho_0$ commutes with $\rho_2,\ldots,\rho_{n-1}$, since the underlying Coxeter diagram is a string. However, $\rho_0$ does not commute with $\rho_1$, since otherwise $G=\langle\rho_0\rangle \times \langle\rho_1,\ldots ,\rho_{n-1}\rangle$; the latter is impossible since $G$ is simple. In a similar way, we can also show that $\rho_{n-1}$ does not commute with $\rho_{n-2}$. Now suppose $n\geq 5$ and consider the subgroup $H:= \langle \rho_0,\rho_1,\rho_{n-2},\rho_{n-1}\rangle$ of $G$. Then $H$ must be isomorphic to $D_{2c}\times D_{2d}$ for some integers $c,d \geq 3$. However, inspection of the list of maximal subgroups of $\Ree(q)$ described above shows that direct products of (non-abelian) dihedral groups never occur as subgroups in $G$. So $n$ is at most 4.
\end{proof}

\subsection{String C-groups of rank four}

Next we eliminate the possibility that the rank is $4$. We begin with a general lemma about string C-groups that are simple.

\begin{lemma}\label{lemma2b}
Let $(G,S)$ be a string C-group of rank $n$, and let $G$ be simple.
Then $N_G(G_{01})\backslash N_G(G_{0})$ must contain an involution (namely $\rho_0$).
\end{lemma}

\begin{proof}
The involution $\rho_0$ centralizes $G_{01}$ and hence must lie in $N_G(G_{01})$. On the other hand, $\rho_0$ cannot also lie in $N_G(G_0)$ for otherwise $G_0$ would have to be a nontrivial normal subgroup in the simple group $G$.
\end{proof}

The next two lemmas will be applied to dihedral subgroups in subgroups of type $\PSL_{2}(q)$ or $C_{2}\times \PSL_{2}(q)$ of $\Ree(q)$, respectively. 

\begin{lemma}
\label{divd}
Let $q=3^{2e+1}$ and $e\geq 0$. Then the order $2d$ of a non-abelian dihedral subgroup of $\PSL_{2}(q)$ must divide $q-1$ or $q+1$. Moreover, $d\not\equiv 0\bmod 4$, and $d$ is even only if $2d$ divides $q+1$.
\end{lemma}

\begin{proof}
Suppose $D_{2d}$ is a non-abelian dihedral subgroup of $\PSL_{2}(q)$, so $d\geq 3$. We claim that $2d$ must divide $q+1$ or $q-1$. Recall that under the assumptions on $q$, the order $2d$ must either be $6$ or must divide $q-1$ or $q+1$. It remains to eliminate $6$ as a possible order. In fact, since $q$ is an odd power of $3$, the only maximal subgroups of $\PSL_{2}(q)$ with an order divisible by $6$ are subgroups $\PSL_{2}(q')$ with $q'$ a smaller odd power of $3$. If we apply this argument over and over again with smaller odd powers of $3$, we eventually are left with a subgroup $\PSL_{2}(3)$. However, $\PSL_{2}(3)\cong A_4$ and hence cannot have a subgroup of order $6$. Thus $2d$ must divide $q+1$ or $q-1$. This proves the first statement of the lemma. The second statement follows from the fact that $q\equiv 3\bmod 8$.
\end{proof}

\begin{lemma}
\label{normc2psl}
Let $q=3^{2e+1}$ and $e\geq 0$, let $2d$ divide $q-1$ or $q+1$, and let $D_{2d}$ be a non-abelian dihedral subgroup of a group $C:=C_{2}\times \PSL_{2}(q)$.\\[.01in]
\indent (a) Then there exists a dihedral subgroup $D$ in $\PSL_{2}(q)$ such that $D_{2d}$ is a subgroup of $C_{2}\times D$ of index $1$ or $2$, and $N_{C}(D_{2d})=N_{C}(C_{2}\times D)=C_{2}\times N_{\PSL_{2}(q)}(D)$. Here the normalizer $N_{\PSL_{2}(q)}(D)$ must lie in a maximal subgroup $D_{q+1}$ or $D_{q-1}$ of $\PSL_{2}(q)$, and coincide with $N_{D_{q+1}}(D)$ or $N_{D_{q-1}}(D)$, according as $2d$ divides $q-1$ or $q+1$. \\[.01in]
\indent (b) Let $D_{2d}\cong D$ (that is, the index is $2$). If $2d\mid (q-1)$ or if $2d\mid (q+1)$ and $(q+1)/2d$ is odd, then $N_{\PSL_{2}(q)}(D)=D$ and $N_{C}(D_{2d})\cong C_2\times D_{2d}$. If $2d\mid (q+1)$ and $(q+1)/2d$ is even, then $N_{\PSL_{2}(q)}(D)\cong D_{4d}$ and $N_{C}(D_{2d})\cong C_2\times D_{4d}$.\\[.01in]
\indent (c) If $D_{2d}=C_{2}\times D$ (that is, $d$ is even, $d/2$ is odd, $D\cong D_d$, and the index is $1$), then $N_{\PSL_{2}(q)}(D)\cong D_{2d}$ and $N_{C}(D_{2d})\cong C_2\times D_{2d}$ (regardless of whether $2d\mid (q-1)$ or $2d\mid (q+1)$).\\[.01in]
\indent (d) The structure of $N_{C}(D_{2d})$ only depends on $d$ and $q$, not on the way in which $D_{2d}$ is embedded in $C$. 
\end{lemma}

\begin{proof}
For the first part, suppose $C_{2}=\langle\rho\rangle$ and $D_{2d}=\langle\sigma_0,\sigma_1\rangle$ where $\sigma_0,\sigma_1$ are standard involutory generators for $D_{2d}$. Write $\sigma_{0}=(\rho^{i},\sigma_{0}')$ and $\sigma_{1}=(\rho^{j},\sigma_{1}')$ for some $i,j=0,1$ and involutions $\sigma_{0}',\sigma_{1}'$ in $\PSL_{2}(p)$. Then $D:= \langle\sigma_{0}',\sigma_{1}'\rangle$ is a dihedral subgroup of $\PSL_{2}(p)$, and $D_{2d}$ lies in $C_{2}\times D$. Since the period of $\sigma_{0}'\sigma_{1}'$ divides that of $\sigma_{0}\sigma_{1}$, the order of $D$ is at most $2d$ and $D_{2d}$ has index $1$ or $2$ in $C_{2}\times D$. If this index is $1$ then $D_{2d}=C_{2}\times D$ (and $d$ is even and $D\cong D_{d}$). If the index of $D_{2d}$ in $C_{2}\times D$ is $2$, then $D\cong D_{2d}$ and $D_{2d}\cap \{1\}\times D$ must have index $1$ or $2$ in $\{1\}\times D$. If the index of $D_{2d}\cap \{1\}\times D$ in $\{1\}\times D$ is $1$ then clearly $D_{2d}=\{1\}\times D$ and $D_{2d}$ can be viewed as a subgroup of $\PSL_{2}(q)$. If the index of $D_{2d}\cap \{1\}\times D$ is $2$, then $D_{2d}\cap \{1\}\times D$ is of the form $\{1\}\times E$ where $E$ is either the cyclic subgroup $C_d$ of $D$, or $d$ is even and $E$ is one of the two dihedral subgroups of $D$ of order $d$. (Note here that $D_{2d}$ cannot itself be a direct product in which one factor is generated by~$\rho$, since $\rho$ cannot lie in $D_{2d}$.)

Next we investigate normalizers. First note that the normalizer of a direct subproduct in a direct product of groups is the direct product of the normalizers of the component groups. Thus $N_{C}(C_{2}\times D)=C_{2}\times N_{\PSL_{2}(q)}(D)$.

We now show that the normalizers in $C$ of the subgroups $D_{2d}$ and $C_{2}\times D$ coincide. There is nothing to prove if $D_{2d}=C_{2}\times D$. Now suppose that $D_{2d}$ has index $2$ in $C_{2}\times D$ and $E$ is as above. Then it is convenient to write $D_{2d}$ in the form 
\begin{equation}
\label{2dunion}
D_{2d} \,=\, (\{1\}\times E)\; \cup\; (\{\rho\}\times (D\backslash E)).
\end{equation}
If $(\alpha,\beta)\in C$ then 
\begin{equation}\label{norma}
(\alpha,\beta)D_{2d}(\alpha,\beta)^{-1}
= (\{1\}\times \beta E \beta^{-1})\, \cup\, (\{\rho\}\times \beta (D\backslash E)\beta^{-1}).
\end{equation}
Now if $(\alpha,\beta)\in N_{C}(D_{2d})$ then the group on the left in (\ref{norma}) is just $D_{2d}$ itself and therefore $\beta E \beta^{-1}=E$ and $\beta (D\backslash E)\beta^{-1}=D\backslash E$. It follows that $\beta$ normalizes both $E$ and $D$, so in particular $(\alpha,\beta)\in N_{C}(C_{2}\times D)$. Hence $N_{C}(D_{2d})\leq N_{C}(C_{2}\times D)$.

Now suppose that $(\alpha,\beta)\in N_{C}(C_{2}\times D)$. Then $\beta$ normalizes $D$. But $\beta E \beta^{-1}$ must be a subgroup of $D$ of index $2$ isomorphic to $E$, and hence $\beta E \beta^{-1}$ and $E$ are either both cyclic or both are dihedral. Clearly, if both subgroups are cyclic then $\beta E \beta^{-1}=E$. However, the case when both subgroups are dihedral is more complicated. First recall that then $d$ must be even. Now the normalizer $N_{\PSL_{2}(q)}(D)$ of the dihedral subgroup $D$ of $\PSL_{2}(q)$ in $\PSL_{2}(q)$ either coincides with $D$ (that is, $D$ is self-normalized), or is a dihedral subgroup containing $D$ as a subgroup of index $2$. We claim that under the assumptions on $q$, the second possibility cannot occur. In fact, in this case the normalizer would have to be a group of order $4d$, and since $d$ is even, its order would have to be divisible by $8$; however, the order of $\PSL_{2}(q)$ is not divisible by $8$ when $q$ is an odd power of $3$, so $\PSL_{2}(q)$ certainly cannot contain a subgroup with an order divisible by $8$. Thus $N_{\PSL_{2}(q)}(D)=D$. But $\beta$ belongs to the normalizer of $D$ in $\PSL_{2}(q)$, so then $\beta$ must lie in $D$. In particular, $\beta E \beta^{-1}=E$ since $E$ is normal in $D$. Thus, in either case we have $\beta E \beta^{-1}=E$, and since $\beta D \beta^{-1}=D$, also $\beta (D\backslash E) \beta^{-1}=D\backslash E$. Hence, (\ref{norma}) shows that $(\alpha,\beta)\in N_{C}(D_{2d})$. Hence also $N_{C}(C_{2}\times D)\leq N_{C}(D_{2d})$. 

To complete the proof of the first part, note that $D$ must lie in a maximal subgroup $D_{q\pm 1}$ of $\PSL_{2}(q)$ and $N_{\PSL_{2}(q)}(D)=N_{D_{q\pm 1}}(D)$. 

The second and third part of the lemma follow from Lemma~\ref{norma} applied to the dihedral subgroup $D$ of $D_{q\pm 1}$. In particular, $D$ is self-normalized in $D_{q\pm 1}$ if $(q\pm 1)/|D|$ is odd, and $N_{\PSL_{2}(q)}(D)$ is a dihedral subgroup of $D_{q\pm 1}$ of order $2|D|$ if $(q\pm 1)/|D|$ is even. Bear in mind that $(q-1)/2$ is odd, and $(q+1)/2$ is even but not divisible by $4$.

To establish the last part of the lemma, note that $N_{C}(D_{2d})\cong C_2\times D_{2d}$, except when $D\cong D_{2d}$, $2d\mid (q+1)$ and $(q+1)/2d$ is even. However, since $q \equiv 3 \bmod 8$, if $2d\mid (q+1)$ and $(q+1)/2d$ is even then $d$ must be odd. In other words, the situation described in the third part of the lemma cannot occur as this would require $d$ to be even. Thus, if $2d\mid (q+1)$ and $(q+1)/2d$ is even, then we are necessarily in the situation described in second part of the lemma, and so necessarily $N_{C}(D_{2d})\cong C_2\times D_{4d}$.
\end{proof}

Our next lemma investigates possible C-subgroups of $G=\Ree(q)$ of rank $3$. The vertex-figure of a putative regular $4$-polytope with automorphism group $G$ would have to be a regular polyhedron with a group of this kind.

\begin{lemma}
\label{rk3subs}
The only proper subgroups of $\Ree(q)$ that could have the structure of a $C$-group of rank $3$ are Ree subgroups $\Ree(q')$ with $q'\neq 3$ or subgroups of the form $\PSL_2(q')$, $C_2\times \PSL_2(q')$, or $\Ree(3)' \cong \PSL_2(8)$.
\end{lemma}

\begin{proof}
It is straightforward to verify that only subgroups of maximal subgroups of $\Ree(q)$ of the second and third type can have the structure of a rank $3$ C-group. Therefore we are left with Ree subgroups $\Ree(q')$ and subgroups of groups $C_2\times \PSL_2(q')$, with $q'$ an odd power of $3$ dividing $q$, as well as subgroups of type $\Ree(3)'\cong \PSL_2(8)$ inside a subgroup $\Ree(3)$. A forward appeal to Lemma~\ref{lemma4} shows that Ree groups $R(q')$ with $q'\neq 3$ do in fact act flag-transitively on polyhedra, and by~\cite{SC94}, so does $\Ree(3)' \cong \PSL_2(8)$. The complete list of subgroups of $\PSL_2(q')$ is available, for instance, in~\cite{LS2005a}. As $q'$ is an odd power of 3, the group $\PSL_2(q')$ does not have subgroups isomorphic to $A_5$, $S_4$, or $\PGL_2(q'')$ for some~$q''$. Hence, none of the subgroups of $\PSL_2(q')$, except for those isomorphic to a group $\PSL_2(q'')$, with $q''$ an odd power of $3$ dividing $q'$ (and hence $q$), admits flag-transitive actions on polyhedra. Now the maximal subgroups of $C_2\times \PSL_{2}(q')$ comprise the factor $\PSL_2(q')$, as well as all subgroups of the form $C_2\times H$ where $H$ is a maximal subgroup of $\PSL_2(q')$ from the following list: 
\[E_{q'}\!:\!C_\frac{q'-1}{2},\; D_{q'-1},\; D_{q'+1},\; \PSL_2(q'').\] 
A subgroup of $C_2\times \PSL_2(q')$ of the form $C_2\times D_{q'-1}$ is isomorphic to $D_{2(q'-1)}$ (since $q'\equiv 3\bmod 8$), so none of its subgroups (including the full subgroup itself) can act regularly on a non-degenerate polyhedron (with no $2$ in the Schl\"afli symbol). Similarly, a subgroup of $C_2\times \PSL_2(q')$ of the form $C_2\times D_{q'+1}$ is isomorphic to $E_4\times D_{(q'+1)/2}$, so again none of its subgroups (including the full subgroup itself) can act regularly on a non-degenerate polyhedron. Finally, a subgroup of $C_2\times \PSL_2(q')$ of the forms $C_2\times (E_{q'}:C_{(q'-1)/2})$ has an order not divisible by~$4$. Hence, as in the two other cases, none of its subgroups (including the full subgroup itself) can act regularly on a non-degenerate polyhedron.

In summary, the only possible candidates for rank $3$ subgroups of $\Ree(q)$ are of the form $\Ree(q')$, $\PSL_2(q')$, $C_2\times \PSL_2(q')$, and $\Ree(3)' \cong \PSL_2(8)$. We can further rule out a subgroup of type $\Ree(3)$, since $\Ree(3)\cong P\Gamma L_2(8)$ is not generated by involutions.
\end{proof}

For a subgroup $B$ in a group $A$ we define $N_{A}^{0}(B):=\langle a\mid a\in N_{A}(B), a^{2}=1\rangle$. If $B$ is generated by involutions then $B\leq N_{A}^{0}(B)\leq N_{A}(B)$.
We first state a lemma that will be useful in several places.

\begin{lemma}~\label{psl28}
Let $H:=\Ree(3)=P\Gamma L_{2}(8)$, and let $D:=D_{2d}$ be a dihedral subgroup of $H$ of order at least 6.
Then $d = 3$, $7$ or $9$, and in all cases $N^0_H(D) = D$.
\end{lemma}
\begin{proof}
Straightforward.
\end{proof}

The following lemma considerably limits the ways in which Ree groups $\Ree(q)$ might be representable as C-groups of rank $4$.

\begin{lemma}~\label{lemma3}
If the group $G:=\Ree(q)$ can be represented as a string C-group of rank 4, then 
\begin{equation}
\label{enzero}
N_{G}^{0}(G_{01})=N^0_{C_G(\rho_0)}(G_{01}).
\end{equation} 
\end{lemma}

\begin{proof}
Suppose that $G$ admits a representation as a string C-group of rank $4$. Thus 
\[G= \langle\rho_0,\rho_1,\rho_2,\rho_3\rangle.\]
Since $\Ree(3)$ is not generated by involutions, we must have $q\neq 3$.

The subgroup $G_{01}=\langle\rho_2,\rho_3\rangle$ is a dihedral subgroup $D_{2d}$ (say) of the centralizer $C_G(\rho_0)$ of $\rho_0$, and $C_G(\rho_0)\cong\langle\rho_0\rangle\times \PSL_{2}(q)$. Here $d\geq 3$, by arguments similar to those used in the proof of Lemma~\ref{lemma2}. Thus
\begin{equation}
\label{d2d}
D_{2d} \cong \langle\rho_2,\rho_3\rangle = G_{01} \leq G_{1}=\langle\rho_0,\rho_2,\rho_3\rangle \leq C_G(\rho_0) \cong C_2\times \PSL_{2}(q).
\end{equation}
By Lemma~\ref{normc2psl} applied to $G_{01}$ and $C_{G}(\rho_0)$, there exists a dihedral subgroup $D$ in the $\PSL_{2}(q)$-factor of $C_G(\rho_0)$ such that $G_{01}$ is a subgroup of $\langle\rho_0\rangle\times D = C_{2}\times D$ of index at most  $2$ and 
\[ N_{C_G(\rho_0)}(G_{01})=N_{C_G(\rho_0)}(C_{2}\times D)=C_{2}\times N_{PSL_{2}(q)}(D).\]
In fact, the proof of Lemma~\ref{normc2psl} shows that this subgroup $C_{2}\times D$ is just given by $G_1$. But $\rho_{0}\not\in G_{01}$, so $G_{01}$ has index $2$ in $C_{2}\times D=G_{1}$, and $D\cong G_{01}\cong D_{2d}$. Then Lemma~\ref{divd}, applied to $D$, shows that $2d$ must divide either $q+1$ or $q-1$. 

The structure of the normalizer $N_{C_G(\rho_0)}(G_{01})$ can be obtained from Lemma~\ref{normc2psl}. In fact, $N_{C_G(\rho_0)}(G_{01})\cong C_2\times D_{2d}$, unless $2d\mid (q+1)$ and $(q+1)/2d$ is even; in the latter case $N_{C_G(\rho_0)}(G_{01})\cong C_2\times D_{4d}$. In particular, $N_{C_G(\rho_0)}(G_{01})$ is generated by involutions and its order is divisible by $4$. We will show that the normalizer of $G_{01}$ in $C_G(\rho_0)$ captures all the information about the full normalizer $N_{G}(G_{01})$ of $G_{01}$ in $G$ that is relevant for us. A key step in the proof is the invariance of the structure of the normalizer of $G_{01}$ in arbitrary subgroups of $G$ of type $C_{2}\times \PSL_{2}(q)$; more precisely, the structure only depends on $d$ and $q$, not on the way in which $G_{01}$ is embedded in a subgroup $C_{2}\times \PSL_{2}(q)$ (see Lemma~\ref{normc2psl}). 

The full normalizer $N_{G}(G_{01})$ of $G_{01}$ in $G$ must certainly contain $N_{C_G(\rho_0)}(G_{01})$ and also have an order divisible by $8$. We claim that all involutions of the full normalizer $N_{G}(G_{01})$ must already lie in $C_{G}(\rho_0)$ and hence in $N_{C_G(\rho_0)}(G_{01})$.

First note that $N_{G}(G_{01})$ must certainly lie in a maximal subgroup $M$ of $G$ and then coincide with $N_{M}(G_{01})$. (Since $G$ is simple, the normalizer of a proper subgroup of $G$ cannot coincide with $G$.) Inspection of the list of maximal subgroups of $G$ shows that only maximal subgroups $M$ of type $R(q')$, $C_{2}\times \PSL_{2}(q)$ or $N_{G}(A_1)$ have an order divisible by $4$. Only those maximal subgroups could perhaps contain $N_{C_G(\rho_0)}(G_{01})$ and hence $N_{G}(G_{01})$. We investigate the three possibilities for $M$ separately.

Suppose $M$ is a group of type $C_{2}\times \PSL_{2}(q)$. Then the invariance of the structure of the normalizer of $G_{01}$ shows that $N_{M}(G_{01})\cong N_{C_G(\rho_0)}(G_{01})$. However, $N_{C_G(\rho_0)}(G_{01})\leq N_{G}(G_{01})$ and $N_{G}(G_{01})=N_{M}(G_{01})$, so this gives $N_{G}(G_{01})=N_{C_G(\rho_0)}(G_{01})$. But $N_{C_G(\rho_0)}(G_{01})$ is generated by involutions, so $N_{C_G(\rho_0)}(G_{01})=N_{C_G(\rho_0)}^{0}(G_{01})$ and (\ref{enzero}) must hold as well.

Let $M$ be a group of type $N_{G}(A_1)\cong (C_{2}^{2}\times D_{(q+1)/2}):C_{3}$ where $A_1$ is a group $C_{(q+1)/4}$ (recall that $(q+1)/4$ is odd). Then 
all involutions of $M$ must lie in its subgroup $K:=C_{2}^{2}\times D_{(q+1)/2}=C_{2}\times D_{q+1}$. In particular, all involutions of $N_{M}(G_{01})$ must lie in $K$ and hence in $N_{K}(G_{01})$; that is, $N_{M}^{0}(G_{01})\leq N_{K}(G_{01})$. Also, $G_{01}$ itself must lie in $K$ and its order $2d$ must divide $q+1$. The subgroup $K$ lies in the centralizer $C$ of the involution generating the $C_2$-factor in the direct product factorization $C_{2}\times D_{q+1}$ for $K$, and $N_{K}(G_{01})\leq N_{C}(G_{01})$. This subgroup $C$ is of type $C_{2}\times \PSL_{2}(q)$, and so again the invariance of the structure of the normalizers implies that $N_{C}(G_{01})\cong N_{C_G(\rho_0)}(G_{01})$. But $N_{G}(G_{01})=N_{M}(G_{01})$ and therefore  
\[N_{C_G(\rho_0)}(G_{01})=N_{C_G(\rho_0)}^{0}(G_{01})\leq N_{G}^{0}(G_{01})=N_{M}^{0}(G_{01})\leq N_{K}(G_{01})\leq N_{C}(G_{01}).\]
Thus $N_{G}^{0}(G_{01})=N_{C_G(\rho_0)}^0(G_{01})$, as required.

Now let $M$ be a Ree group $\Ree(q')$ where $(q')^p= q$ and $p$ is a prime. We first cover the case when $M$ is a Ree group $\Ree(3)=\PSL_{2}(8):C_3$, that is, $q=3^p$ where $p$ is a prime.
In that case, by Lemma~\ref{psl28}, $N^0_G(G_{01}) = N^0_M(G_{01}) = G_{01}$. 
Hence, since also $G_{01}\leq N_{C_G(\rho_0)}^0(G_{01})\leq N^0_G(G_{01})$, we must have
$N^0_G(G_{01})\leq N_{C_G(\rho_0)}^0(G_{01})$.
  
Now suppose $q'\neq 3$, so in particular $M$ is simple. Then $2d$ must divide $q'\pm 1$, since $N_{C_G(\rho_0)}(G_{01})$ lies in $M$ and therefore $\rho_0\in M$, giving $G_{01}\leq N_{C_M(\rho_0)}(G_{01})\cong C_{2}\times \PSL_{2}(q')$. Since the subgroup $N_{G}(G_{01})$ of $M$ must have an order divisible by $4$, it must lie in a maximal subgroup $M'$ of $M$ of type $\Ree(q'')$, $C_{2}\times \PSL_{2}(q'')$, or $N_{\Ree(q')}(A_1')$ with $A_{1}'\cong C_{(q'+1)/4}$. The maximal subgroups $M'$ of $M=\Ree(q')$ of types $C_{2}\times \PSL_{2}(q')$ and $N_{\Ree(q')}(A_{1}')$, respectively, lie in maximal subgroups of $G$ of type $C_{2}\times \PSL_{2}(q)$ or $N_{G}(A_{1})$, so they are subsumed under the previous discussion. (Alternatively we could dispose of these cases for $M'$ directly, using arguments very similar to those in the two previous cases for $M$.) Then this leaves the possibility that $M'$ is of type $R(q'')$, in which case we are back at a Ree group. Now continuing in this fashion to smaller and smaller Ree subgroups that could perhaps contain $N_{G}(G_{01})$, we eventually arrive at either a Ree subgroup $M^{(k)}$ (say) whose parameter $q^{(k)}\pm 1$ (say) is no longer divisible by $2d$, or a Ree group $\Ree(3)$. In the first case, $\Ree(q')$ does not contribute anything new to $N_{G}^{0}(G_{01})$, and the normalizer $N_{G}(G_{01})$ must already lie in one of the maximal subgroups of type $C_{2}\times \PSL_{2}(q)$ or $N_{G}(A_1)$ discussed earlier; in particular, $N_{G}^{0}(G_{01})=N_{C_G(\rho_0)}^{0}(G_{01})$, as required. In the second case, the normalizer  $N_{G}(G_{01})$ lies in a Ree subgroup $\Ree(3)\cong \PSL_{2}(8):C_3$, and its involutory part $N_{G}^0(G_{01})$ must lie in the $\PSL_{2}(8)$ subgroup. 
Again, by Lemma~\ref{psl28}, we have $N^0_G(G_{01}) = N^0_M(G_{01}) = G_{01}$, hence (\ref{enzero}) must also hold in this case.
\end{proof}

\begin{lemma}~\label{lemma3a}
If the group $G:=\Ree(q)$ can be represented as a string C-group of rank 4, then $q\neq 3$ and both the facet stabilizer $G_3$ and vertex stabilizer $G_0$ have to be isomorphic to $\PSL_2(8)=\Ree(3)'$ (i.e. the commutator subgroup of $\Ree(3)$) or a simple Ree group $\Ree(q')$ with $q = q'^m$ for some odd integer $m$. 
\end{lemma}

\begin{proof}
We consider the possible choices for $G_0$ in the given C-group representation of $G$ of rank $4$. Our goal is to use Lemma~\ref{lemma2b} to limit the choices for $G_0$ to just $\Ree(3)'$ or $R(q')$. First recall from Lemma~\ref{rk3subs} that the only possible candidates for $G_0$ are either Ree subgroups $\Ree(q')$ with $q'\neq 3$ or subgroups of the form $\PSL_2(q')$, $C_2\times \PSL_2(q')$, or $\Ree(3)' \cong \PSL_2(8)$. To complete the proof we must eliminate the second and third types of candidates. This is accomplished by means of  Lemmas~\ref{lemma2b} and \ref{normc2psl}, proving in each case that $N_G(G_{01})\backslash N_G(G_{0})$ cannot contain an involution, or equivalently $N_G^0(G_{01})\leq N_G(G_{0})$. Bear in mind that $G_{01}\leq G_0$.

First observe that all subgroups of $G$ of the form $C_2\times \PSL_2(q')$ are self-normalized in $G$; and the normalizer of a subgroup of $G$ of the form $\PSL_2(q')$ is isomorphic to $C_2\times \PSL_2(q')$. In other words, $N_{G}(G_0)=G_{0}$ if $G_0$ is of type $C_2\times \PSL_2(q')$, and $N_{G}(G_0)=C_{2}\times G_{0}$ if $G_0$ is of type $\PSL_2(q')$. We show that $N_{G}^{0}(G_{01})\leq N_{G}(G_{0})$ for each of these three choices of $G_0$. 

Suppose that $G_0\cong C_2\times \PSL_2(q')$. We first claim that then $2d\mid q'\pm 1$. To see this, note that the intersection of $G_{01}$ with the $\PSL_{2}(q')$-factor of $G_0$ is a subgroup of index $1$ or $2$ in $G_{01}$. If the index is $1$, the statement is clear by Lemma~\ref{divd}, since then $G_{01}$ lies in the $\PSL_{2}(q')$-factor; and if the index is $2$ and the intersection is a cyclic group $C_d$, the statement follows by inspection of the possible orders of cyclic subgroups of $\PSL_{2}(q')$. Now if the index is $2$ and the intersection is a dihedral group $D_d$, then Lemma~\ref{normc2psl} shows that $d$ must be even, $2d\mid q+1$, and $d/2$ must be odd; moreover, $d\mid q'+1$ since $D_d$ lies in $\PSL_{2}(q')$, and hence $2d\mid q'+1$ since $q'+1$ is divisible by $4$. Thus $2d\mid q'\pm 1$, as claimed.

Now, since $G_0\cong C_2\times \PSL_2(q')$, the normalizer $N_{C_{G}(\rho_0)}(G_{01})$ coincides with the normalizer $N_{H}(G_{01})$ of $G_{01}$ taken in a suitable subgroup $H$ of $C_{G}(\rho_0)$ of type $C_{2}\times \PSL_{2}(q')$. In fact, from Lemma~\ref{normc2psl} we know that 
\[ N_{C_{G}(\rho_0)}(G_{01})\leq C_{2}\times D_{q\pm 1}\leq C_{G}(\rho_0)\cong C_2\times \PSL_2(q).\] 
But $2d\mid q'\pm 1$, so we must have $N_{C_{G}(\rho_0)}(G_{01})\leq C_{2}\times D_{q'\pm 1}$.  However, $C_{2}\times D_{q'\pm 1}$ lies in a subgroup $H$ of $C_{G}(\rho_0)$ isomorphic to $C_2\times \PSL_2(q')$.

To complete the argument (for any given type of group $G_0$) we show that $N_{G}^{0}(G_{01})$ must lie in $N_{G_0}(G_{01})$ and therefore also in $G_{0}$ and $N_{G}(G_{0})$. When $G_0$ is a group of type $C_2\times \PSL_2(q')$, the normalizer $N_{G_0}(G_{01})$ can be determined using Lemma~\ref{normc2psl} (with $q$ replaced by $q'$). In fact, by the invariance of the normalizers of $G_{01}$ we know that $N_{G_0}(G_{01})$ and $N_{H}(G_{01})$ are isomorphic and that both subgroups are generated by involutions. However, then by Lemma~\ref{lemma3},
\[ N_{G_0}(G_{01}) = N_{G_0}^{0}(G_{01})\leq N_{G}^{0}(G_{01})=N_{C_{G}(\rho_0)}^0(G_{01}) = N_{H}^0(G_{01}) = N_{H}(G_{01}), \]
so clearly $N_{G_0}(G_{01})=N_{H}(G_{01})$. Thus $N_{G}^{0}(G_{01})=N_{G_0}(G_{01})\leq G_{0}\leq N_{G}(G_{0})$.

Now let $G_0$ be of type $\PSL_2(q')$. Then $C:=N_{G}(G_0)$ is a group of type $C_{2}\times \PSL_{2}(q')$ containing $G_0$, so 
we can replace $G_0$ by $C$ and argue as before. In fact, using the same subgroup $H$, we see that the normalizers $N_{C}(G_{01})$ and $N_{H}(G_{01})$ are isomorphic subgroups generated by involutions. In particular,
\[ N_{C}(G_{01}) = N_{C}^{0}(G_{01})\leq N_{G}^{0}(G_{01})=N_{C_{G}(\rho_0)}^0(G_{01})=N_{H}^0(G_{01})=N_{H}(G_{01}), \]
and therefore $N_{C}(G_{01})=N_{H}(G_{01})$. Hence $N_{G}^{0}(G_{01})=N_{C}(G_{01})\leq C=N_{G}(G_{0})$.
\end{proof}

Let us now show that $G_0 \not\cong R(3)'$.
\begin{lemma}~\label{lemma3b}
If $\Ree(q)$ has a representation as a string C-group of rank 4 with $G_0 \cong R(3)'$, then $q=27$.
\end{lemma}
\begin{proof}
Suppose $G:=\Ree(q)$ is represented as a string C-group of rank $4$ with generators $\rho_0,\ldots,\rho_3$. Then we know that $G_{01} \leq G_1 \leq C_G(\rho_0) \cong C_2\times \PSL_2(q)$. 

The abstract regular polyhedra with automorphism group $R(3)'=\PSL_2(8)$ are all known and are available, for instance,  in~\cite{LV2005}. There are seven examples, up to isomorphism, but not all can occur in the present context. In fact, the dihedral subgroup $G_{01}$ of $G_0$ must also lie $C_G(\rho_0)\cong C_2\times \PSL_2(q)$ and hence cannot be a subgroup $D_{18}$. 
It follows that the polyhedron associated with $G_0$ (that is, the vertex-figure of the polytope for $G$) must have Schl\"afli symbol $\{3,7\}$, $\{7,3\}$, $\{7,7\}$, $\{7,9\}$ or $\{9,7\}$. 
We can further rule out the possibility that $G_{01}\cong D_6$ or $D_{18}$ by Lemmas~\ref{divd} and \ref{normc2psl}, giving that $C_2\times \PSL_{2}(q)$ has no dihedral subgroup of order $6$ or $18$.
Hence $G_{01} \cong D_{14}$.

The fixed point set of every involution in $G$ is a block of the corresponding Steiner system $S(2,q+1,q^3+1)$, and vice versa, every block is the fixed point set of a unique involution. Hence, two involutions with two common fixed points must coincide, since their blocks of fixed points must coincide. Suppose $B_0$ denotes the block of fixed points of $\rho_0$. As $\rho_2$ and $\rho_3$ centralize $\rho_0$, they stabilize $B_0$ globally. However, $\rho_2$ cannot have a fixed point among the $q+1$ points in $B_0$, since otherwise two points of $B_0$ would have to be fixed by $\rho_2$ since $q+1$ is even. Thus $\rho_2$, and similarly $\rho_3$, does not fix any point in $B_0$. Moreover, in order for $G_{01}\cong D_{14}$ to lie in a subgroup of $G$ of type $C_2\times \PSL_2(q)$, we must have $7\mid q+1$ or $7\mid q-1$. Using $q=3^{2e+1}$ and working modulo $7$ the latter possibility is easily seen to be impossible. On the other hand, the former possibility occurs precisely when $e\equiv 1 \bmod 3$, and then $3\mid 2e+1$. Hence $G$ must have subgroups isomorphic to $\Ree(27)=\Ree(3^3)$.

We claim that $G$ itself is isomorphic to $\Ree(27)$, that is, $q=27$. Now the subgroup $G_0 \cong R(3)'$ lies in a unique subgroup $K\cong \Ree(3)$ of $G$, namely its normalizer $N_G(G_0)$. This subgroup $K$, in turn, lies in a unique subgroup $H\cong \Ree(27)$ of $G$. All Ree subgroups of $G$ are self-normalized in $G$, so in particular $K$ and $H$ are self-normalized. Relative to the Ree subgroup $H$, the normalizer $N_{H}(C_7)$ in $H$ of the cyclic subgroup $C_7$ of $G_{01}$ is a maximal subgroup of type $N_{H}(A_1) = (C_2^2\times D_{14}):C_3$ in $H$, which also contains $G_{01}$ (see Section~\ref{reebasics} or \cite[p. 123]{Con85}). Note here that this subgroup $C_7$ is a 7-Sylow subgroup of both $K$ and $H$, and is normalized by $G_{01}$. Thus, $N_H(C_7) = (C_2^2\times D_{14}):C_3$. We claim that  $N_H(G_{01})=N_H(C_7)$. Clearly, $N_H(G_{01})\leq N_H(C_7)$. For the opposite inclusion observe that $(C_2^2\times D_{14}):C_3$ has four subgroups isomorphic to $D_{14}$, including $G_{01}$. The subgroup $G_{01}$ is normalized by the $C_{3}$-factor, and the three others are permuted under conjugation by $C_3$. Hence, among these four subgroups only $G_{01}$ is normal and can be thought of as the subgroup $D_{14}$ occurring in the factorization of the semi-direct product. It follows that the subgroups $C_{2}^2$ and $C_3$ normalize $G_{01}$. Thus 
\[ N_H(G_{01})=N_H(C_7) = (C_2^2\times D_{14}):C_3 . \]

Figure~\ref{sublat} shows the sublattice of the subgroup lattice of $G$ that is relevant to the current situation. Each box contains two pieces of information:\ a group that describes the abstract structure of the groups in the conjugacy class of subgroups of $G$ depicted by the box, and a number in the lower left corner that gives the number of subgroups in the conjugacy class. This number is the order of $G$ divided by the order of the normalizer in $G$ of a representative subgroup of the conjugacy class. Two boxes are joined by an edge provided that the subgroups represented by the lower box are subgroups of some subgroups represented by the upper box. There are also two numbers on each edge. The  number at the top gives the number of subgroups in the conjugacy class for the lower box that are contained in a {\sl given\/} subgroup in the conjugacy class for the upper box. The number at the bottom similarly is the number of subgroups in the conjugacy class for the upper box that contain a {\sl given\/} subgroup in the conjugacy class for the lower box.
If we know the lengths of the conjugacy classes for the upper box and lower box, then knowing one of these two numbers on the connecting edge gives us the other. For instance, in Figure~\ref{sublat}, if we know that there are 36 (conjugate) subgroups $D_{14}$ in a given subgroup $R(3)'$, then there are 
\[ {\frac{|G|}{|R(3)|}\,.\,36} \,\slash\, {\frac{|G|}{|2^2\,.\,3\,.\,14|}} = 4\] 
(conjugate) subgroups $R(3)'$ containing a given subgroup $D_{14}$.

Returning to our line of argument, Figure~\ref{sublat} tells us that $G_0\cong R(3)'$ is in a unique subgroup isomorphic to $R(27)$, namely $H$ (because of the lower 1's on the edges joining the boxes). It also shows that $G_{01}$ is contained in a unique subgroup $(C^2_2\times D_{14}):C_3$, which, in turn, is contained in a unique $R(27)$, namely $H$. As we saw above, this subgroup $(C^2_2\times D_{14}):C_3$ is necessarily the normalizer $N_{H}(G_{01})$ of $G_{01}$ in $H$. Moreover, $\rho_0$ has to lie in this unique subgroup $(C^2_2\times D_{14}):C_3$, which itself is a subgroup of $H$, and therefore $\langle \rho_0,G_0\rangle \leq H$. This holds because $N_G(G_{01}) = N_H(G_{01})$. 
That these normalizers coincide can be seen as follows. Clearly, $N_H(G_{01}) \leq N_G(G_{01})$. Now for the opposite inclusion observe that for $g\in N_G(G_{01})$ we have $G_{01} = gG_{01}g^{-1} \leq g N_H(G_{01})g^{-1}$ and (trivially) $G_{01}\leq N_H(G_{01})$.
But then Figure~\ref{sublat} shows that a subgroup $D_{14}$ of $H$ must lie in a unique conjugate of $(C^2_2\times D_{14}):C_3=N_{H}(G_{01})$, so necessarily $g N_H(G_{01})g^{-1} = N_H(G_{01})$. Similarly, since $N_H(G_{01})\leq H$ and hence $N_H(G_{01})=gN_H(G_{01})g^{-1}\leq gHg^{-1}$, Figure~\ref{sublat} (at box $R(27)$) gives $gHg^{-1}=H$, so $g\in H$ since $H$ is self-normalized.
Thus $G = \langle \rho_0,G_0\rangle = H \cong R(27)$. 
\end{proof}

\begin{figure}[htbp]
\begin{picture}(20,220)
\put(0,10){
\put(-18,200){\framebox{$\Ree(q)$}}
\put(0,194){\line(0,-1){32}}
\put(-25,150){\framebox{$_{\frac{|G|}{|\Ree(27)|}}\Ree(27)$}}
\put(0,138){\line(0,-1){26}}
\put(-25,100){\framebox{$_{\frac{|G|}{|\Ree(3)|}}\Ree(3)$}}
\put(0,88){\line(0,-1){26}}
\put(-35,50){\framebox{$_{\frac{|G|}{|\Ree(3)|}}\Ree(3)' \cong \PSL_2(8)$}}
\put(0,38){\line(0,-1){26}}
\put(-25,0){\framebox{$_{\frac{|G|}{|2^2.3.14|}}D_{14}$}}
\put(-55,129){\tiny $|\Ree(27)|/|\Ree(3)|$}
\put(2,115){\tiny 1}
\put(2,79){\tiny 1}
\put(2,65){\tiny 1}
\put(2,79){\tiny 1}
\put(2,65){\tiny 1}
\put(2,29){\tiny 36}
\put(2,15){\tiny 4}
\put(110,75){\framebox{$_{\frac{|G|}{|2^2.3.14|}}(C_2^2\times D_{14}):C_3$}}
\put(31,11){\line(2,1){100}}
\put(130,50){\tiny 1}
\put(35,5){\tiny 1}
\put(110,90){\tiny 1}
\put(59,128){\tiny $|\Ree(27)|/168$}
\put(37,138){\line(2,-1){100}}}
\end{picture}
\caption{A sublattice of the subgroup lattice of $\Ree(q)$}
\label{sublat}
\end{figure}
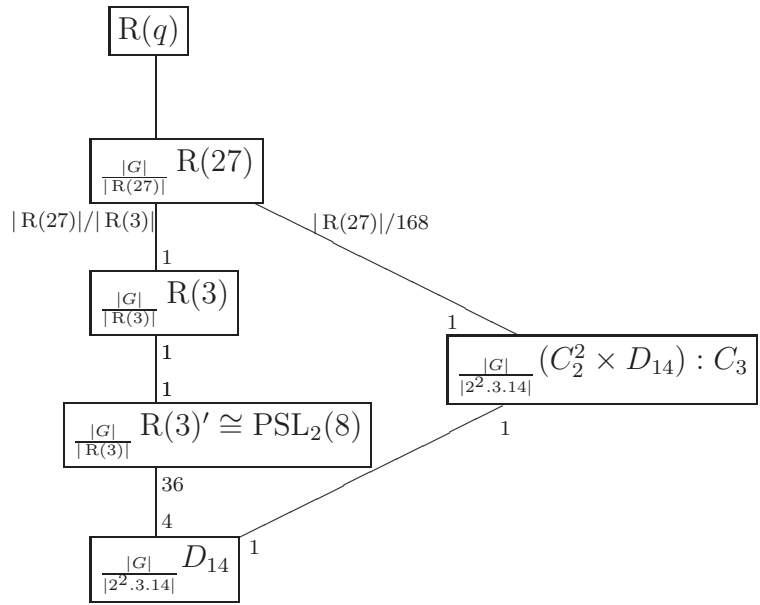

\begin{lemma}~\label{lemma3c}
The group $\Ree(27)$ cannot be represented as a string C-group of rank 4.
\end{lemma}

\begin{proof}
Let $G\cong \Ree(27)$.
By the previous lemmas we may assume that $G_0\cong G_3\cong \PSL_2(8)$. In all other cases we know that $G$ cannot be represented as a rank 4 string C-group.
Moreover, from the proof of the previous lemma we already know that $G_{01} \cong D_{14}$ and 
$N_{G}(G_{01}) \cong (C_2^2\times D_{14}) : C_3$. As there is a unique conjugacy class of subgroups $\Ree(3)'$ in $\Ree(27)$, and there is also a unique conjugacy class of subgroups $D_{14}$ in $\Ree(3)'$, the choice of $\rho_2,\rho_3$ is therefore unique up to conjugacy in $\Ree(27)$. Once $\rho_2,\rho_3$ have been chosen, there are three candidates for $\rho_0$, namely the elements of the subgroup $C_2^2$ that centralizes $D_{14}$, and these are equivalent under conjugacy by $C_3$. Hence there is a unique choice for $\{\rho_0, \rho_2, \rho_3\}$ up to conjugacy. By duality we also know that $G_3 \cong \Ree(3)'$ and $G_{23}\cong D_{14}$, and that the pair $(G_{23},G_3)$ is related to $(G_{01},G_{0})$ by conjugacy in $R(27)$. Hence there must exist an element $g\in \Ree(27)$ such that 
\begin{itemize} 
\item $\rho_0^g = \rho_3$, $\rho_2^g = \rho_1$, $\rho_3^g=\rho_0$,\ or
\item $\rho_0^g = \rho_3$, $\rho_2^g = \rho_0$, $\rho_3^g=\rho_1$.
\end{itemize}
The second case can be reduced to the first, as the centraliser of $\rho_0$ contains an element that swaps $\rho_2$ and $\rho_3$ (any two involutions in $D_{14}$ are conjugate).
Hence, we may assume without loss of generality that $g$ swaps $\rho_0$ and $\rho_3$. In particular, $\langle \rho_0, \rho_3 \rangle$ is an elementary abelian group of order 4 normalized by $g$. All such subgroups are known to be conjugate and have as normalizer a group $(C_2^2\times D_{14}):C_3$. In this group, there is no element that will swap $\rho_0$ and $\rho_3$ under conjugation. All elements that will conjugate $\rho_0$ to $\rho_3$ will necessarily conjugate $\rho_3$ to $\rho_0\rho_3$. Hence we have a contradiction.
\end{proof}

We therefore know that if a string C-group representation of rank 4 exists for $\Ree(q)$, both $G_0$ and $G_3$ must be subgroups of Ree type. Thus from now on we can assume $G_0\cong \Ree(q')$ with $q'>3$.

In a Ree group, the dihedral subgroups $D_{2n}$ are such that $n$ must divide one of 
\[ 9,\ q-1,\ q+1,\ \aq := q+1-3^{e+1},\ \bq := q+1+3^{e+1}. \]
Note that 
\[ \aq\bq= q^{2}-q+1, \]
so in particular if $H$ is a Ree subgroup $\Ree(q')$ of $G$ then similarly $\aqp\bqp = q'^2-q'+1$.

\begin{lemma}\label{newLemma1}
Let $G \cong \Ree(q)$ with $q=3^{2e+1}$ and $\langle \rho_0, \rho_1, \rho_2, \rho_3\rangle$ be a string C-group representation of rank 4 of $G$.
Then
\begin{enumerate}
\item $G_{01}$ is a dihedral subgroup $D_{2d}$ with $d$ a divisor of $q+1$ and of either $\aqp$ or $\bqp$ for some $q'$ such that $q = q'^m$ with $m$ odd (where $q'$ is determined by $G_{0}=\Ree(q')$);
\item $m = 3$, and $G_0$ and $G_3$ are conjugate Ree subgroups $\Ree(q')$ with $q = q'^3$;
\item $G_{03}$ is a dihedral subgroup $D_{2t}$ with $t$ a divisor of $\aqp$ or $\bqp$. 
\end{enumerate}
\end{lemma}
\begin{proof}
(1) By Lemmas~\ref{lemma3a},~\ref{lemma3b} and~\ref{lemma3c}, we may assume that $G_0$ is a simple Ree subgroup of $G$. Let $G_0$ be a Ree subgroup $\Ree(q')$, with $q'\neq 3$ such that $q'^m = q$ with $m$ a positive odd integer and let $G_{01}\cong D_{2d}$. 
As $G_{01} \leq C_G(\rho_0)$ we have that $2d\mid q\pm 1$ by Lemma~\ref{divd}. 
In order to have involutions in $N_{G}(G_{01})\backslash N_{G}(G_{0})$, the only possibility is that $N_{G}(G_{01})$ (of order divisible by $4$) lies in a maximal subgroup of type $N_{G}(A_1)$ but not in a maximal subgroup $N_{G_0}(C_{\frac{q'+1}{4}})$ of $G_0$; for otherwise, the same techniques as in Lemma~\ref{lemma3a} show that there is no involution in $N_{G}(G_{01})\backslash N_{G}(G_{0})$. Hence $2d$ divides $q+1$. Observe that $N_{G}(A_1)\cong (C_2^2\times D_{\frac{q+1}{2}}):C_3$ has exactly four subgroups $D_{\frac{q+1}{2}}$ because of the subgroup $C_2^2$. These four subgroups are not all normalised by the $C_3$ because of the semi-direct product. Hence the $C_3$ must conjugate three of them and normalise the fourth one. Similarly, in $R(q')$ there are four subgroups $D_{q'+1}$ in each $N_{R(q')}(A_1')$ and it is obvious that $N_{R(q')}(D_{2d}) = N_{R(q)}(D_{2d})$ for every divisor $d$ of $q'+1$. Hence, in order to find some involutions in $N_{G}(G_{01})\backslash N_{G}(G_{0})$, we need to have that $2d$ does not divide $q'+1$. Moreover,  since $q'-1$ divides $q-1$, we have also that $(q'-1,q+1) = 2$. That forces $d$ not to be a divisor of $q'-1$ as $d>2$. Hence, looking at the list of maximal subgroups of $R(q')$ we can conclude that $d$ is a divisor of either $\alpha_{q'}$ or $\beta_{q'}$ in order for $D_{2d}$ to be a dihedral subgroup of $R(q')$. 

(2)  Observe that $q'^3+1 = (q'+1)\alpha_{q'}\beta_{q'}$ divides $q^3+1$. 
Let us first show that $2e+1$ must be divisible by $3$ in order for $d$ to satisfy (1).
Suppose $(3,2e+1) = 1$. Then $q' = 3^{2f+1}$ with $2e+1 = m(2f+1)$ and $(3,m) = 1$.
Let $p$ be an odd prime dividing $(\alpha_{q'}\beta_{q'},q+1)$ but not dividing $q'+1$.
Then $p$ divides $(q'^3+1,q+1)$ and hence $p$ divides 
\[ (q'^6-1,q'^{2m}-1) = q'^{2(3,m)}-1 = q'^2-1 = (q'+1)(q'-1)\] 
and hence also $q'-1$. As $p$ divides $q+1$, and $q'-1$ divides $q-1$, and since $(q-1,q+1) = 2$, we have that $p\mid 2$, a contradiction. Hence $m$ must be divisible by $3$ and so does $2e+1$.
Suppose $m \neq 3$. Then $m = 3m'$ and given a Ree subgroup $\Ree(q')$ of $\Ree(q)$ with $q'^m = q$, there exists a Ree subgroup $\Ree(q'^3)$ such that $\Ree(q') < \Ree(q'^3) < \Ree(q)$.
Using similar arguments as in the proof of Lemma~\ref{lemma3b}, it is easy to show that, since $\alpha_{q'}\beta_{q'}$ divides $q'^3+1$, we must have $\langle \rho_0, \rho_1, \rho_2, \rho_3 \rangle = \Ree(q'^3)$ and therefore $m=3$. Indeed, as we stated in (1), $N_{R(q'^3)}(D_{2d}) = N_{R(q)}(D_{2d})$ for every divisor $d$ of $q'^3+1$. Hence $\rho_0 \in \Ree(q'^3)$. This implies that
 $m=3$ and $G_0 \cong \Ree(q')$ with $q'^3 = q$. Dually, $G_3 \cong \Ree(q')$. As all subgroups $\Ree(q')$ are conjugate in $\Ree(q)$, we have that $G_0$ and $G_3$ are conjugate.

(3) is due to the fact that $G_0 \cap G_3 = G_{03}$ and that, by (2), $G_0$ and $G_3$ are conjugate in $G$. Hence, $N_G(G_{03}) \setminus G_0$ has to be nonempty and $G_{03}$ must not be contained in a subgroup $H$ of $G_0$ such that $N_G(H) \geq N_G(G_{03})$, for if such a subgroup $H$ exists, then $G_0 \cap G_3 \geq H$. If $t$ divides $9$ or one of $q'\pm 1$, this does not happen. Hence $t$ divides one of $\alpha_{q'}$ or $\beta_{q'}$.
\end{proof}

\begin{lemma}
The small Ree groups have no string C-group representation of rank $4$.
\end{lemma}
\begin{proof}
Suppose $G$ is a Ree group that has a string C-group representation of rank $4$. By Lemma~\ref{lemma3a} and part (2) of Lemma~\ref{newLemma1} we may assume that $G:=\Ree(q)$ where $q=q'^3$ with $q'=3^{m}$ for an odd integer $m$. Moreover, $G_0$ and $G_3$ are conjugate simple Ree subgroups isomorphic to $\Ree(q')$.
By part (3) of Lemma~\ref{newLemma1}, if $G_{03}=D_{2t}$ then $t$ must be a divisor of either $\aqp$ or $\bqp$, and since $q = q'^3$, we also have 
\[ q+1 = (q'+1)(q'^2-q'+1) = (q'+1)\aqp\bqp.\] 
Thus $t$ is also a divisor of $q+1$. 
We claim that then $G_0 \cap G_3 > G_{03}$, which gives a contradiction to the intersection property.
Indeed, since $G_{03}$ lies in a subgroup $H:=C_{t}:C_6$ of $G_0$, and the normaliser of $G_{03}$ is not contained in $G_0$ (for otherwise, $D_{2t}$ would have to lie in a unique subgroup $\Ree(q')$, whereas already $G_0$ and $G_3$ give two examples of such subgroups, by the previous lemma), we have $N_G(G_{03}) = (E_4\times D_{2t}) : C_3$. This group contains $H= C_t:C_6 \cong D_{2t} : C_3 $ as a normal subgroup, and $G_{03}$ is normal in $H$.
We also have that $N_G(H) = N_G(G_{03})$. But then, as $G_{03}$ is normal in $H$, any subgroup $R(q')$ containing $G_{03}$ must contain $H$. In particular this applies to $G_3$. Thus $G_0 \cap G_3 \geq H > G_{03}$, and the intersection property fails.

\end{proof}

\subsection{String C-groups of rank $3$}

It remains to investigate the possibility of representing $\Ree(q)$ as a string C-group of rank $3$. The following lemma gives an affirmative answer and completes the proof of Theorem~\ref{maintheo}.

\begin{lemma}~\label{lemma4}
Let $G=\Ree(q)$, with $q\neq 3$ an odd power of 3. Then there exists a triple of involutions $S := \{\rho_0,\rho_1,\rho_2\}$ in $G$ such that $(G,S)$ is a string C-group.
\end{lemma}

\begin{proof}
Recall that the fixed point set of an involution in $G$ is a block of the Steiner system ${\cal S} := S(2,q+1,q^3+1)$. Pick two involutions $\rho_0,\rho_1$ from a maximal subgroup $M$ of $G$ of type $N_{G}(A_3)$ such that $\rho_0\rho_1$ has order $q+1+3^{e+1}$, and let $B_0,B_1$, respectively, denote their blocks of fixed points. Obviously, $B_{0}\cap B_{1} = \emptyset$, for otherwise $\langle \rho_0,\rho_1\rangle$ would lie in the stabilizer of a point in $B_{0}\cap B_{1}$, which is not possible because of the order of $\rho_0\rho_1$. Recall here that the point stabilizers are maximal subgroups of the form $N_{G}(A)=A:C_{q-1}$, where $A$ is a $3$-Sylow subgroup of $G$. Now choose an involution $\rho_2$ in $C_G(\rho_0)$ distinct from $\rho_0$ such that its block of fixed points $B_2$ meets $B_{1}$ in a point. Then $B_{1}\cap B_{2}$ must consist of a single point $p$ (say), and $B_{0} \cap B_{2} = \emptyset$ since the stabilizer of a point does not contain Klein 4-groups. Then $\langle \rho_1,\rho_2 \rangle$ lies in the point stabilizer of $p$, and hence must a dihedral group $D_{2n}$, with $n$ a power of 3. As $\langle\rho_{0},\rho_{1}\rangle$ is a subgroup of index $3$ in $M$, and $\rho_0$ does not belong to $M$, we see that $\langle\rho_0,\rho_1,\rho_2\rangle = G$. Moreover, since the orders of $\rho_0\rho_1$ and $\rho_1\rho_2$ are coprime, the intersection property must hold as well. Thus $(G,S)$, with $S:=\{\rho_0,\rho_1,\rho_2\}$, is a string C-group of rank~$3$.
\end{proof}

We have not attempted to enumerate or classify all representations of $\Ree(q)$ as a string $C$-group of rank $3$.

\section{Acknowledgements}
This research was sponsored by a Marsden Grant (12-UOA-083) of the Royal Society of New Zealand.

\bibliographystyle{plain}

\end{document}